\documentclass[english,a4paper]{amsart}

\usepackage{amssymb}
 \usepackage{amsthm}
 
 \usepackage{ulem}

\usepackage{lineno}

\usepackage{amsmath}

\usepackage{mathrsfs,amsfonts,dsfont} 
\usepackage{mathrsfs}
\usepackage{bbm}
\usepackage{amsfonts}
 \usepackage{dsfont}
 \usepackage{graphicx}
\usepackage{mathrsfs}
\usepackage{wrapfig}
\usepackage{tikz}
\usetikzlibrary{arrows,positioning,shapes.geometric}
\usepackage{amsrefs}
\usepackage{babel}
\usepackage[latin1]{inputenc}
\usepackage{eurosym}

\def\bE{\mathbb E}
\def\bP{\mathbb P}
\def\bR{\mathbb R}
\def\ss{\mathbb S}

\def\cA{{\mathcal A}}

\def\cD{{\mathcal D}}
\def\cR{{\mathcal R}}

\def\X {X:=\{X(s),s\in\ss\}}

\newtheorem{Def}{Definition}
\newtheorem{Th}{Theorem}[section]

\newtheorem{Co}[Th]{Corollary}
\newtheorem{Pro}[Th]{Proposition}
\newtheorem{Lem}[Th]{Lemma}

\newtheorem{remark}{Remark}

\numberwithin{equation}{section}

 \definecolor{violetb}{rgb}{0.7, 0.2, 0.5}

 \newcommand{\vero}[1]{{\color{black}#1}}
  
   \newcommand{\veron}[1]{{\color{black}#1}}

\newcommand{\cel}[1]{{\color{black}#1}}

\title{Spatial risk measure for \veron{max-stable and max-mixture} processes}
\author{M. Ahmed}
\address{Universit\'e de Lyon, Universit\'e Lyon 1, Institut Camille Jordan ICJ UMR 5208 CNRS, France\\
Department of statistics, University of Mosul, Iraq}
\author{ V. Maume-Deschamps}
\address{Universit\'e de Lyon, Universit\'e Lyon 1, Institut Camille Jordan ICJ UMR 5208 CNRS, France}
\author{P.Ribereau}
\address{Universit\'e de Lyon, Universit\'e Lyon 1, Institut Camille Jordan ICJ UMR 5208 CNRS, France}
\author{C.Vial}
\address{Universit\'e de Lyon, Universit\'e Lyon 1, Institut Camille Jordan ICJ UMR 5208 CNRS, France}
\cel{\address{INRIA, Villeurbanne, France}}

\begin{document}




\begin{abstract}%
\vero{In this paper, we consider isotropic and stationary max-stable, inverse max-stable and max-mixture processes $X=(X(s))_{s\in\bR^2}$ and the damage function $\cD_X^{\nu}= |X|^\nu$ with $0<\nu<1/2$. We study the quantitative behavior of a risk measure which is the variance of the average of $\cD_X^{\nu}$ over a region $\mathcal{A}\subset \bR^2$.} This kind of risk measure has already been introduced and studied for \vero{some} max-stable processes in \cite{koch2015spatial}. 
We evaluated the proposed risk measure by a simulation study.
\end{abstract}
\maketitle




\section{Introduction}\label{I:11}
Storms are the most destructive natural hazards in Europe. The economic  and the private sectors losses due to these extreme events are often \vero{important}. For example, during December 1999, three storms hit Europe causing insured losses above 10 billion \vero{\euro ~(see \cite{ulbrich2001three,re2001winterstorms,donat2011high})}. 
\vero{The storms may have a huge spatial component; in other words, the underlying spatial process may have a strong spatial dependence even at \veron{a} long distance.}\\ 
\\ 
One of the main \vero{characteristics of} climate  events is \vero{the} spatial dependence. \vero{Many dependence structures may arise: Asymptotic  dependence; Asymptotic  independence or both \cite{wadsworth2012dependence}}. The high impact of storm losses motivated us  to \vero{propose risk measures taking into account the spatial dependence.}\\
\\
In case \vero{of} univariate random variables, risk measures has been widely studied in \vero{the} literature and the corresponding axiomatic formulation has been presented in \cite{artzner1999coherent}. \vero{In \cite{follmer2014spatial} a collection of risk measures indexed by a network is introduced for some financial products. In spatial contexts,  the spatial dependence plays an important role. For example, wind speed and rainfall amount e.g. have different spatial behavior, so that, after normalization of their marginal distributions, the value of a risk measure should not be the same. }\\
In \cite{keef2009spatial}, the authors proposed to evaluate the risk on a region by a probability \veron{$\bP(S>s)$} where $S$ is an integrated damage function. In  \cite{koch2015spatial} or \cite{KOCHErwan2014tools} this idea is developed to define  spatial risk measures taking into account the spatial dependence. \vero{In \cite{ahmed2016spatial} the same idea is used: a risk measure constructed with the damage function $\cD_{X\/,u} = (X-u)^+$ with $u$ a fixed threshold for a Gaussian process $X$ is studied. In the same spirit as \cite{artzner1999coherent}, the authors propose a set of axioms that a risk measure in the spatial context should verify. This point of view has been previously adopted in \cite{koch2015spatial} for some } max-stable processes. 
Our main contributions \vero{concern the risk measure based  on the intensity damage function $\cD_X^{\nu}= |X|^\nu$ with  $0<\nu<1/2$, it consists in the development of the results from \cite{koch2015spatial}: further max-stable processes are involved and the computation technics are extended to max-mixture processes. We study the properties of the risk measure with respect to the parameters of each model (with a focus on the dependence parameter). We also study its axiomatic properties.} \\
This paper is organized as follows. \vero{Section \ref{II}  recalls definitions and properties of max-stable and max-mixture processes. In Section \ref{R:2} we consider spatial risk measures and recall the axiomatic setting from \cite{ahmed2016spatial} which derives from \cite{koch2015spatial}.} Section \ref{R:4} is devoted to the study of the risk measures with damage function $|X|^\nu$ for max-stable and max-mixture processes. We propose forms of this risk measures and derive its behavior. We present in Section \ref{R:5} a simulation study in order to evaluate this spatial risk measures. Concluding remarks are discussed in Section \ref{R:7}.
\section{Spatial extreme  processes}\label{II}
\vero{We shall focus on max-stable processes, inverse max-stable processes and max-mixtures of both, and we call these processes extreme processes. We shall emphasize on the modelization of the dependence structure and thus assume that the marginal laws have been normalized to unit Fréchet with distribution function $F(x)=\exp(-1/x)$, $x>0$ (see \cite{naveau2009modelling}).}
\subsection{Max-stable model}\label{max-stable}
\vero{This is an} extension of the multivariate extreme value theory to \veron{the spatial} setting. \vero{We refer to \cite{de1984spectral,de2007extreme} for definitions and properties of max-stable processes. We shall consider max-stable processes on $\ss\subset \bR^2$  with unit Fréchet marginal distributions (i.e. simple max-stable processes), for any $(s\/,t)\in \ss^2$,}
\begin{equation}
\begin{split}
\bP\big(X(s)\leq x_1, X(t)\leq x_2\big)=&G_{s,t}(x_1,x_2)\\
=&\exp(-V_{s,t}(x_1,x_2))\/,
\end{split}
\end{equation}
\vero{where $V_{s\/,t}$ \veron{is the} so-called exponent measure function. \veron{It} is homogenous of order $-1$ and satisfies  the bounds  }
  \begin{equation}\label{max47}
\max\{1/x_1,1/x_2\}\leq V_{s,t}(x_1,,x_2)\leq \{1/x_1+1/x_2\}.
\end{equation}
\vero{These bounds imply that $X$ is positive quadrant dependent (PQD),  see \cite{lehmann1966some} for definitions and properties of PQD processes. In this paper, we consider stationary and isotropic processes. Thus, the exponent measure $V_{s\/,t}$ and the distribution function $G_{s\/,t}$ depend only on the \veron{norm} $h=||s-t||$ and will be denoted \veron{by} $V_h$ and $G_h$.}\\
 \ \\
In \cite{de1984spectral}, it is also proved that every simple max-stable process $X$ has the following  spectral representation:
\begin{equation*}\label{max2BB}
X(s)=\max_{i\geq1}\xi_iW_i(s)\quad s\in\ss\/,
\end{equation*}
where $\{\xi_i,i\geq1\}$ is an i.i.d Poisson point process on $(0,\infty)$, with intensity $d\xi/\xi^2$ and $\{W_i,i\geq1\}$ are i.i.d  copies of a positive random field $W=\{W(s),s\in\ss\}$, such that $\bE[W(s)]=1$ for all $s$ and independent of $\xi_i$. \\
\ \\
Many dependence measures \vero{ for spatial processes $X$ have been introduced. These are generally bivariate dependence measures used in a spatial context. The tail dependence coefficient $\chi$ introduced in \cite{ledford1996statistics} is defined by}
 \begin{equation}\label{max49}
\chi(h)=\lim_{u\to1}\bP\big(F(X(s))>u|F(X(s+h))>u\big)\/.
\end{equation}
\vero{
If $\chi(h)=0$, \vero{the pair $(X(s+h)\/,X(s))$} is said to be asymptotically independent (AI).\\
If $\chi(h)\neq0$, \vero{the pair $(X(s+h)\/,X(s))$}  is said to be asymptotically dependent (AD).\\
The process is said AI (resp. AD) if for all $h\in\ss$ $\chi(h)=0$  (resp. $\chi(h)\neq0$).}
\vero{The extremal coefficient $\Theta(h) = V_h(1\/,1)$ satisfies $\chi(h) =2-\Theta(h)$ and (see \cite{wadsworth2012dependence}) }
\begin{equation}\label{max48}
G_h(x,x)=\exp(-\Theta(h)/x).
\end{equation}  
\vero{In \cite{coles1999dependence} an alternative definition of the tail dependence coefficient is given. }
 \begin{equation}\label{max51}
\chi(h,u)=2-\frac{\log\bP\big(F(X(s))<u,F(X(s+h))<u\big)}{\log\bP\big(F(X(s))<u\big)},\quad0\leq u\leq1.
\end{equation}
\vero{We have} $\lim_{u\to 1}\chi(h,u)=\chi(h)$. \\
\\
\vero{The spectral representation is useful to construct specific max-stable processes. We present \vero{three} of them: Smith, Schlater and truncated Schlater models.}\\
\ \\
\textbf{Smith Model}
\veron{Introduced in \cite{smith1990max}.}  It is defined on $\ss=\bR^d$. Its dependence structure is contained in a covariance matrix $\Sigma$. Let $\{(\xi_i,s_i)\}$ be \vero{a Poisson point }process on $(0,\infty)\times\bR^d$ with intensity $\xi^{-2}\mathrm{d}\xi\mathrm{d}s$ and  consider the $d$-dimensional Gaussian probability density function $\varphi_d(.;\Sigma)$ with mean 0 and covariance matrix $\Sigma$. For all $s \in \bR^d$, \vero{define}  $W_i(s)=\varphi_d(s-s_i;\Sigma)$ and
 \begin{equation*}
X(s)=\max_{i\geq1}\{\xi_i \varphi_d(s-s_i;\Sigma)\}.
\end{equation*}
The exponent measure function is given by
\begin{equation*}
V_h(x_1,x_2)=\frac{1}{x_1}\Phi\bigg(\frac{\tau(h)}{2}+\frac{1}{\tau(h)}\log\frac{x_2}{x_1}\bigg)+\frac{1}{x_2}\Phi\bigg(\frac{\tau(h)}{2}+\frac{1}{\tau(h)}\log\frac{x_1}{x_2}\bigg);
\end{equation*} 
\veron{with} $\tau(h)=\sqrt{h^T\Sigma^{-1}h}$ and $\Phi(\cdot)$ the standard normal cumulative distribution function.
 
 The pairwise extremal coefficient equals
  \begin{equation*}
\Theta(h)=2\Phi\bigg(\frac{\tau(h)}{2}\bigg).
\end{equation*}  

Note that if  the covariance  matrix is diagonal $\Sigma=\sigma \bf I_d$, then the process $X$ is isotropic as its bivariate distribution depends only on $h$ through the function $\tau(h)=\frac{1}{\sigma}\|h\|$.\\
 
\ \\
\textbf{Schlather Models}
This model introduced \veron{in} \cite{schlather2002models} provides a class based on \veron{a} stationary \veron{Gaussian}  random field. Let $W:=\{W(s),s\in\ss\}$ be a stationary random field, with $\bE\big[W^+(s)\big]=\mu\in(0,\infty)$ where $W^+(s)=\max\{0,W(s)\}$. Let $\{\xi_i,i\geq1\}$ be a Poisson point  process on $(0,\infty)$, with intensity $\mathrm{d}\xi/\xi^2$ and $\{W_i,i\geq1\}$ are iid copies of $W(s)$. Consider
\begin{equation*}
X(s)=\mu^{-1}\max_{i\geq1}\xi_iW^+_i(s),\quad s\in\ss\/,
\end{equation*} 
it defines a stationary max-stable process. Schlather proposed \veron{to take} a stationary Gaussian process  $W(s)$ with correlation function $\rho(\cdot)$ and $\mu^{-1}=\sqrt{2\pi}$. In this case, the resulting max-stable process $X$ is called \textit{ Extremal Gaussian process (EG)}. The exponent measure function is
\begin{equation*}
V_h(x_1,x_2)=\frac{1}{2}\bigg(\frac{1}{x_1}+\frac{1}{x_2}\bigg)\bigg[1+\sqrt{1-2(\rho(h)+1)\frac{x_1x_2}{(x_1+x_2)^2}}\bigg].
\end{equation*} 
The extremal coefficient is given by
 \begin{equation*}
\Theta(h)=1+\bigg(\frac{1-\rho(h)}{2}\bigg)^{1/2}.
\end{equation*} 
\vero{We have  $\lim_{h\to\infty}\chi(h)\neq 0$.} In other words, the asymptotic dependence persists even at infinite distances. This might be unrealistic in  applications. To overcome this problem a truncated version of $W(s)$ can be used. Let $\{r_i\}$ be a  homogenous  Poisson point process of unit rate on $\ss$ and $\mu^{-1}=\sqrt{2\pi}(\bE[|\mathcal{B}|])^{-1}$. Then, for a stationary Gaussian process $W_i(s)$, define
 \begin{equation}\label{max57}
X(s)=\max_{i\geq1}\xi_iW_i(s)\mathds{1}_{\mathcal{B}_i}(s-r_i),\quad s\in\ss
\end{equation} 
\vero{with $\mathcal{B} \subset \ss$ a compact random set and $\mathcal{B}_i$ i.i.d. copies of $\mathcal{B}$.}
\vero{The process $X$} \textit{is a truncated extremal Gaussian process (TEG)}.  The exponent measure function is given by 
\begin{equation*}
V_h(x_1,x_2)=\bigg(\frac{1}{x_1}+\frac{1}{x_2}\bigg)\bigg[1-\frac{\alpha(h)}{2}\bigg(1-\sqrt{1-2(\rho(h)+1)\frac{x_1x_2}{(x_1+x_2)^2}}\bigg)\bigg].
\end{equation*} 
The extremal coefficient is given by
 \begin{equation*}
\Theta(h)=2-\alpha(h)\left\{1-\bigg(\frac{1-\rho(h)}{2}\bigg)^{1/2}\right\}
\end{equation*} 
where $\alpha(h)=\bE\{|\mathcal{B}\cap(h+\mathcal{B})|\}/\bE[|\mathcal{B}|]$.\\
\\
\vero{\veron{Usually,} $\mathcal{B}$ \veron{is} a disk of radius $r$}. \veron{In that case,} $\alpha(h)=\{1-h/2r\}_+$.  For more details, see \cite{davison2012geostatistics}. That leads to $\chi(h)=0,\forall h\geq2r$. In other word the process $X$ is  \veron{asymptotically independent (and thus independent because it is max-stable)} for all $h\geq2r$. 
\subsection{Inverse max-stable processes}
\vero{Max-stable processes are either AD or they are independent. This behavior may be unapropriate in applications: data may reveal asymptotic independence without being independent. }
\\ 
\vero{In \cite{coles1999dependence} the \textbf{lower tail dependence coefficient $\overline{\chi}(h)$}  is proposed in order to study the strengt of dependence in AI cases.}
\begin{equation}\label{ind2}
\overline{\chi}(h)=\lim_{u\to1}\frac{2\log\bP\big(F(X(s))>u\big)}{\log\bP\big(F(X(s))>u,F(Y(s+h))>u\big)} \veron{-1}\/,\quad 0\leq u\leq1.
\end{equation} 
 
\vero{We have $-1\leq \overline{\chi}(h)\leq 1$ and the spatial process is asymptotically dependent if $\overline{\chi}(h)=1$. Otherwise, it is asymptotically independent.}
\\
\vero{We shall consider a class of asymptotically independent processes introduced in   \cite{wadsworth2012dependence}: \textbf{Inverse max-stable process}. Let $X'$ be a max-stable process with unit Fréchet margin, consider} 
$$
 X(s)=g(X(s))=-1/\log\{1-e^{-1/X'(s)}\}\quad s\in \ss\/.
$$
\vero{Then $X$ is asymptotically independent with unit Fréchet margin and bivariate  survivor function }
$$\bP\big(X(s_1)>x_1,X(s+h)>x_2\big)=\exp\big(-V_h\big(g(x_1),g(x_2)\big)\big).
$$
where $V_h$ is the exponent \vero{measure function of $X'$. We a slight langage abuse, we shall say that $V_h$ is the exponent measure function of $X$.} 
%
\vero{Inverse max-stable processes enter in the class of processes defined in \cite{ledford1996statistics} which satisfy:
$$ \bP\big(X(s)>x,X(s+h)>x\big)=\mathcal{L}_h(x)x^{-1/\eta(h)},\quad x\to\infty
$$
where $\mathcal{L}_h(x)$ is \veron{a} slowly varying function and $\eta(h)\in(0,1]$ is called \textbf{ the tail dependence coefficient}. For these kind of processes, the AI is caracterized by $\eta(h)<1$.}
\subsection{Max-Mixture model}
\vero{In spatial contexts, specifically in environmental domain many scenarios of dependence could arise and AD and AI might cohabite. 
The  work by \cite{wadsworth2012dependence} provides a flexible model called max-mixture}.\\
\vero{Let $X$  be a max-stable process, with extremal coefficient $\Theta(h)$ and exponent measure function $V_h^X$. Let $Y$ be an inverse max-stable process with tail dependence coefficient $\eta(h)$ and exponent measure function $V_h^Y$}. Assume that $X$ and $Y$ are independent and each of them has Fréchet margin. Let  $a\in[0,1]$ and define
$$Z(s)=\max\{aX(s),(1-a)Y(s)\}, \quad s\in\ss\/.
$$
 \vero{$Z$ has unit  Fréchet marginals. Its bivariate distribution function is given by}
 \begin{equation}\label{max22}
\bP\big(Z(s)\leq z_1,Z(s+h)\leq z_2\big)=e^{-aV^X_h(z_1,z_2)}\bigg[e^{\frac{-(1-a)}{z_1}}+e^{\frac{-(1-a)}{z_2}}-1+e^{-V^Y_h(g_a(z_1),g_a(z_2))}\bigg]\/,
\end{equation} 
\vero{where  $g_a(z) = g(\frac{z}{1-a})$.
Its bivariate survivor function satisfies}
$$\bP\big(Z(s)>z,Z(t)>z\big)\sim\frac{a\{2-\Theta(h)\}}{z}+\frac{(1-a)^{1/\eta(h)}}{z^{1/\eta{(h)}}}+O(z^{-2}),\quad z\to\infty.
$$
\vero{If  $h^*=\inf\{h:\Theta(h)\neq 0\}<\infty$, then $Z$ is asymptotically dependent up to distance $h^*$ and asymptotically independent for larger distances.} See \cite{bacro2016flexible} for more details. \veron{Of course, if $a=0$ then $Z$ is indeed an inverted max-stable process. If $a=1$ then $Z$ is a max-stable process. }
 Moreover   
\begin{equation}\label{co1}
\chi(h)=a(2-\Theta(h))
\end{equation}
and 
\begin{equation}\label{co2}
\overline{\chi}(h)=\mathds{1}_{[h^*< h]}(h)+(2\eta(h)-1)\mathds{1}_{[h*\geq h)}.
\end{equation}
\section{Spatial risk measures. }\label{R:2}
Consider a \vero{spatial process $\X$, $\ss\subset\bR^2$. We use the definition of risk measures proposed in  \cite{ahmed2016spatial} and \cite{koch2015spatial}. Given a damage function $\mathcal{D}~:\cR^2~\longrightarrow~\cR^+$, and
 $\cA\in\mathcal{B}(\bR^d)$,  the normalized aggregate loss function on $\mathcal{A}$ is
\begin{equation*}
  L(\cA,\mathcal{D})=\frac{1}{|\cA|}\int_{\cA}\mathcal{D}(s)\/\mathrm{d}s,
\end{equation*}
{where $| \cA|$ stands for the volume of $\cA$.}
}
The quantity $\displaystyle\int_{\cA}\mathcal{D}(s)\mathrm{d}s$ represents the aggregated loss over the region $\cA$. Therefore the function $ L(\cA,\mathcal{D})$ is the proportion of loss on a $\cA$.
\subsection{Definition of spatial risk measures.}\label{R} 
In this paper, \vero{we work with unit Fréchet margin processes and thus $(X-u)^+$ as no finite expectation nor variance. The risk measure considered in \cite{ahmed2016spatial} is not suitable. In \cite{koch2015spatial}, the damage function $\mathbf{1}_{\{X>u\}}$ is considered and the subsequent risk measure is computed for  Smith, Schlater and the so-called tube processes. This damage function does not take into account the behaviour of the process over the threshold $u$, this is why, we did not consider it. We shall consider the damage function 
\begin{equation*}
  \mathcal{D}^\nu_{X}(s)=|X(s)|^\nu\/,
\end{equation*}
for $0<\nu<\frac12$. This type of damage function is used e.g. in analyzing  the negative effects due to the wind speed (see  \cite{re2013natural} for more details). In  \cite{koch2015spatial}, the risk measure associated to $\mathcal{D}^\nu_X$ has been computed for Smith processes.
}
%

\vero{Since we work with stationary processes, the expectation of the normalized loss function do not take into account the dependence structure. As in  \cite{ahmed2016spatial} and \cite{koch2015spatial}, we shall focus on its variance. 

$$ \cR_1(\mathcal{A},\mathcal{D}_X)\} = \mathrm{Var}\big(L(\mathcal{A},\mathcal{D}_X)\big)\/.$$ 
If $X$ is a spatial process with unit Fréchet marginal distributions, $\cR_1(\mathcal{A},\mathcal{D}_X^\nu)$ is well defined provided that $0<\nu<\frac12$.\\
Remark that}
\begin{equation}\label{SRE-risk1b}
\cR_1(\mathcal{A},\mathcal{D}_X)\}=\frac{1}{|\cA|^2}\int_{\cA\times\cA}\mathrm{Cov}\big(\cD_X(s),\cD_X(t)\big)\mathrm{d}s\mathrm{d}t.
\end{equation} 
\subsection{Axiomatic properties of spatial risk measures.}\label{R:3}
Several authors such as \cite{artzner1999coherent}, \cite{krokhmal2007higher} and \cite{tsanakas2003risk} presented \vero{an axiomatic setting for univariate risk measures. In \cite{koch2015spatial} a first set of axioms for risk measures in spatial context is considered for} the damage functions: $\cD_X(s) = \mathbf{1}_{\{X(s)>u\}}$, $\cD_X(s)= X(s)^\nu$ \vero{where $X$ is a max-stable process. In   \cite{ahmed2016spatial} the damage function  $\cD_X(s) = (X(s)-u)^+$ for Gaussian processes has been investigated.} 
\\
\ \\
\vero{Let us recall the mains axioms proposed in \cite{koch2015spatial} and \cite{ahmed2016spatial} \veron{for  the real valued spatial risk measure $\cR_1(\cA,\cD)$.}} 
Axioms 1. and 4. below have been introduced in \cite{koch2015spatial}, and studied for some max-stable processes. 
\begin{Def}\label{def:axioms}
Let $\cA\subset\bR^2$ be a region of the space. 
\begin{enumerate}
\item \textbf{Spatial invariance under translation}\\
Let $\cA+v\subset\bR^2$ be the region $\cA$ translated by a vector $v\in\bR^2$. Then for $v\in\bR^2$, {$\cR_1(\cA+v,\cD)=\cR_1(\cA,\cD)$}.

\item \textbf{{Spatial anti-monotoncity}}\\
 Let $\cA_1$,$\cA_2\subset\bR^2$, two regions such that { $|\cA_1|\leq|\cA_2|$}, then {$\cR_1(\cA_2,\cD)\leq\cR_1(\cA_1,\cD)$}.
\item \textbf{{Spatial sub-additivity}}\\
Let $\cA_1$,$\cA_2\subset\bR^2$ be two regions disjointed, then {$\cR_1(\cA_1\cup\cA_2,\cD)\leq\cR_1(\cA_1,\cD)+\cR_1(\cA_2,\cD)$}.
\item \textbf{Spatial super sub-additivity}\\
Let $\cA_1$,$\cA_2\subset\bR^2$ be two regions disjointed, then {$\cR_1(\cA_1\cup\cA_2,\cD)\leq\min_{i=1\/,2}\left[\cR_1(\cA_i,\cD)\right]$}.
\item \textbf{{Spatial homogeneity}}\\
Let $\lambda>0 $ and $\cA\subset \cR^2$ then {$\cR_1(\lambda\cA,\cD)=\lambda^k\cR_1(\cA,\cD)$, that is $\cR_1$ }is homogenous of order $k$, {where $\lambda \cA$ is the set $\{\lambda x, x\in\cA\}$.}

\end{enumerate}
\end{Def}

\vero{In  \cite{koch2015spatial} the invariance by translation, the monotonicity and super sub-additivity in the case where $\cA_1\/, \cA_2$ are either disks or squares is proved for max-stable processes for the damage function $\mathbf{1}_{\{X>u\}}$ and for the damage function $X^\nu$ in the case of the Smith process. While in \cite{ahmed2016spatial}  the invariance by translation and  sub-additivity is proved for any processes provided that $\cD_X$ admits an order $2$ moment. The anti-monotonicity \veron{for disks and squares} is proved for the damage function $(X-u)^+$ with $X$ a Gaussian process. We shall study further the properties of $\cR_1(\mathcal{A},\mathcal{D}_X)$ for max-mixture processes.}

\section{Risk measures for max-mixture processes.}\label{R:4}
\vero{Let  $X$ an isotropic and stationary process, with unit Fréchet margin, let $0<\nu<1/2$ be a fixed.}
\subsection{General forms for $\cR_1(\mathcal{A},\mathcal{D}_{X}^\nu)$}
\vero{The following result shows that the computation of $\cR_1(\mathcal{A},\mathcal{D}_X)$ may reduce to smaller dimension integral. It has been proved in  \cite{KOCHErwan2014tools} for Smith models. Following the lines of its proof, it remains valid provided that the damage function $X^\nu$ has an order $2$ moment (see Theorem 3.3 in \cite{ahmed2016spatial}). \\
\ \\Let $f_{disk}(\cdot,R)$ and $f_{square}(\cdot,R)$ be the density of the distance between two points randomly chosen in a disk of radius $R$ and a square of side $R$ respectively. 
We have (see \cite{moltchanov2012distance}):
\begin{equation*}
f_{disk}(h,R)=\frac{2h}{R^2}\bigg(\frac{2}{\pi}\mathrm{{arccos}}\big(\frac{h}{2R}\big)-\frac{h}{\pi R}\sqrt{1-\frac{h^2}{4R^2}}\bigg),
\end{equation*}
and
\begin{equation*}
f_{square}(h,R)= \frac{2\pi h}{R^2}-\frac{8h^2}{R^3}+\frac{2h^3}{R^4}
\end{equation*}
where $b=\frac{h^2}{R^2}$.
}
 \begin{Lem}\label{disk-max}
\vero{ Let $\X$ be an isotropic and stationary spatial process such that the damage function $\cD_X$ has finite order $2$ moment. \\
Let $\mathcal{Q}(h)=\mathrm{Cov}\big(\cD_X(s),\cD_X(s+h)\big)$.\\
Consider $\cA\subset \bR^2$ a disk of radius $R$, we have:
\begin{equation}\label{Th-disk}
\cR_1(\cA,\cD_X)=\mathrm{Var}\big(L(\mathcal{A},\mathcal{D}_X)\big)=\int_{h=0}^{2R}\mathcal{Q}(h) f_{disk}(h,R)\mathrm{d}h,
\end{equation}
Consider $\cA\subset \bR^2$ a  square of side $R$, we have:\\
\begin{equation}\label{Th-square}
\cR_1(\cA,\cD_X)=\mathrm{Var}\big(L(\mathcal{A},\mathcal{D}_X)\big)=\int_{h=0}^{\sqrt{2}R}\mathcal{Q}(h) f_{square}(h,R)\mathrm{d}h,
\end{equation}
}
 \end{Lem}

In what follows, \vero{ results are written for square regions $\cA$}, but the results hold for disks as well.
\vero{\begin{remark}\label{Rm:1}
Properties of moments of Fréchet distributions give that if $X$ has unit Fréchet marginal distributions, 
$$\mathbb{E}(L(\cA\/,\cD_X^\nu)) = \Gamma(1-\nu)\/.$$
\end{remark}}

\begin{Pro}\label{hoff}
Consider $\X$ an isotropic and stationary spatial  process with unit Fréchet margin $F$ and  pairwise distribution function $G_h^X=\bP(X(s)\leq x_1,X(s+h)\leq x_2)$. Let $\cA$ \veron{be} a square of side $R$. We have
\begin{equation}\label{CTEG}
\cR_1( \cA,\cD_X^{\nu})=\int_{h=0}^{\sqrt{2}R} \mathcal{Q}(h,\nu)f_{square}(h,R)\mathrm{d}h, 
\end{equation}
with 
\begin{equation*}
\mathcal{Q}(h,\nu)=\mathrm{Cov}\big(\cD_X^{\nu}(s),\cD_X^{\nu}(s+h)\big);
\end{equation*} 
\begin{equation}\label{TEG-t}
\mathcal{Q}(h,\nu)=\int_{0}^{\infty}\int_{0}^{\infty}\big[G_h^X(x_1^{1/\nu},x_2^{1/\nu})-F(x_1^{1/\nu})F(x_2^{1/\nu})\big]\mathrm{d}x_1\mathrm{d}x_2
\end{equation}
or equivalently
\begin{equation}\label{teg-QQ}
\mathcal{Q}(h,\nu) =\nu^2\int_{0}^{\infty}\int_{0}^{\infty}x_1^{\nu-1}x_2^{\nu-1}\big[G_h^X(x_1,x_2)-F(x_1)F(x_2)\big]\mathrm{d}x_1\mathrm{d}x_2.
\end{equation}
\end{Pro} 
\begin{proof}
\vero{ Since $X$ is a non negative process, the result follows directly from Hoeffding's identity (\cite{hougaard2012analysis}  and \cite{sen1994impact}): }
 \begin{eqnarray*}
 \lefteqn{\mathrm{Cov}\big(\cD_X^{\nu}(s),\cD_X^{\nu}(s+h)\big)}\\
& =&\iint_{\bR_+^2}\big[\bP\big(X(s)^\nu\leq x_1,X(s+h)^\nu\leq x_2\big)\\
&& -\bP\big(X(s)^\nu\leq x_1\big)\bP\big(X(s+h)^\nu\leq x_2\big)\big]\mathrm{d}x_1\mathrm{d}x_2 \\
&=&\nu^2\iint_{\bR^+} x_1^{\nu-1}x_2^{\nu-1}\big[\bP\big(X(s)\leq x_1,X(s+h)\leq x_2\big)\\
&& -\bP\big(X(s)\leq x_1\big)\bP\big(X(s+h)\leq x_2\big)\big]\mathrm{d}x_1\mathrm{d}x_2.
\end{eqnarray*}
 \end{proof}
 \subsection{Explicit form \veron{for} $\cR_1( \cA,\cD_X^{\nu})$ for TEG max-stable process $X$}
\vero{Equation (\ref{CTEG}) shows that, \veron{if $\cA$ is either a disk or a square}, the computation of  $\cR_1( \cA,\cD_X^{\nu})$ reduces to the integration of $\mathcal{Q}(h\/,\nu) f_{square}$ (resp. $\mathcal{Q}(h\/,\nu) f_{disk}$). In \cite{KOCHErwan2014tools}, the computation of $\mathcal{Q}(h\/,\nu) f_{square}$ for the Smith model has been done. In that case, the computation of  $\cR_1( \cA,\cD_X^{\nu})$ is  reduced to a one dimensional integration. In this section, we do the computation for a TEG model.}
 \begin{Co}\label{TEG}
 \textcolor{black}{
Let $\X$ be a truncated extremal Gaussian TEG max-stable process with unit Fréchet margin, correlation function $\rho$ and truncated parameter $r$. For  $0<\nu<1/2$, we have 
\begin{equation*}\begin{split}
\mathcal{Q}(h,\nu) =&\\
\int_0^{+\infty}w^{\nu}\bigg[&\Gamma{(2(1-\nu))}\mathcal{T}_2(w,h)\mathcal{T}_1(w,h)^{2(\nu-1)}+\Gamma{(1-2\nu)}\mathcal{T}_3(w,h)\mathcal{T}_1(w,h)^{2\nu-1}\bigg]\mathrm{d}w\\
-&\big[\Gamma(1-\nu)\big]^2
\end{split}
\end{equation*}
where, 
\begin{equation}\label{TT:1}
\mathcal{T}_1(w,h)=\frac{w+1}{w}\bigg[1-\frac{\alpha(h)}{2}\big(1-\mathcal{K}(w,h)\big)\bigg];
\end{equation}
\\
\begin{equation}\label{TT:2}
\begin{split}
\mathcal{T}_2(w,h)=&\bigg[1-\frac{\alpha(h)}{2}\big(1-\mathcal{K}(w,h)\big)-\frac{\alpha(h)(\rho(h)+1)(1-w)}{2\mathcal{K}(w,h)(w+1)^2}\bigg]\\
\\
\times &\bigg[\frac{1}{w^2}-\frac{\alpha(h)}{2w^2}\big(1-\mathcal{K}(w,h)\big)-\frac{\alpha(h)(\rho(h)+1)(w-1)}{2w\mathcal{K}(w,h)(w+1)^2}\bigg];
\end{split}
\end{equation}
\\
\begin{equation}\label{TT:3}
\mathcal{T}_3(w,h)=\alpha(h)\bigg[\frac{(\rho(h)+1)}{\mathcal{K}(w,h)(w+1)^3}-\frac{(\rho(h)+1)^2(w-1)^2}{2\mathcal{K}(w,h)^3(w+1)^5}\bigg];
\end{equation}
\\
$$\mathcal{K}(w,h)=\bigg[1-\frac{2w(\rho(h)+1)}{(w+1)^2}\bigg]^{1/2} $$
and $\alpha(h)=\{1-\frac{h}{2r}\}_+$.
}
\end{Co}
\begin{proof}
\textcolor{black}{
We have, 
\begin{equation*}
 \mathrm{Cov}\big(\cD_X^{\nu}(s),\cD_X^{\nu}(s+h)\big)=\bE\big[\cD_X^{\nu}(s)\cD_X^{\nu}(s+h)\big]-\big[\bE[\cD_X^{\nu}(s)]\big]^2.
 \end{equation*}
  From Remark \ref{Rm:1}, $\bE[\cD_X^{\nu}(s)]=\Gamma(1-\nu)$. \veron{Moreover}, 
 \begin{equation*}\label{TEG1}
 \bE\big[\cD_X^{\nu}(s)\cD_X^{\nu}(s+h)\big]=\int_{0}^{\infty}\int_{0}^{\infty}x_1^{\nu}x_2^{\nu}f_{(X(s),X(s+h))}(x_1,x_2)\mathrm{d}x_1\mathrm{d}x_2,
 \end{equation*}
where $f_{(X(s),X(s+h))}(x_1,x_2)$ is the bivariate  density function of the TEG model.}
\vero{It rewrites:}
\textcolor{black}{\begin{equation*}\label{TEG3}
\bE[\mathcal{D}^\nu(s)\mathcal{D}^\nu(s+h)]=\int_{0}^{+\infty}\int_{0}^{+\infty}u^{2\nu+1}w^{\nu}f(u,uw)\mathrm{d}u\mathrm{d}w.
\end{equation*}
}
\vero{The bivariate density function of a TEG model \veron{is given by}}
\textcolor{black}{\begin{equation*}\label{TEG4}
f_{(X(s),X(s+h))}(u,uw)=\big[\frac{1}{u^4}\mathcal{T}_2(w,h)+\frac{1}{u^3}\mathcal{T}_3(w,h)\big]e^{\frac{-1}{u}\mathcal{T}_1(w,h)}
\end{equation*}
where $\mathcal{T}_1(w,h)$, $\mathcal{T}_2(w,h)$ and $\mathcal{T}_3(w,h)$ are \veron{given}  in (\ref{TT:1}), (\ref{TT:2}) and (\ref{TT:3}). Therefore 
\begin{equation*}\label{TEG5}
\begin{split}
\bE[\mathcal{D}^\nu(s)\mathcal{D}^\nu(s+h)]=&\\
\int_{0}^{+\infty}w^{\nu}\bigg[\mathcal{T}_2(w,h)\int_{0}^{+\infty}&u^{2\nu-3}e^{\frac{-1}{u}\mathcal{T}_1(w,h)}\mathrm{d}u+\mathcal{T}_3(w,h)\int_{0}^{+\infty}u^{2\nu-2}e^{\frac{-1}{u}\mathcal{T}_1(w,h)}\mathrm{d}u\bigg]\mathrm{d}w.
\end{split}
\end{equation*}
Moment properties of Fréchet distributions give}
\vero{\begin{equation*}\label{TEG6}
\int_{0}^{+\infty}u^{2\nu-3}e^{\frac{-1}{u}\mathcal{T}_1(w,h)}\mathrm{d}u=\frac{1}{\mathcal{T}_1(w,h)}\/.\mu_{(2\nu-1)},
\end{equation*}
}
with $\mu_{(2\nu-1)}$  the moment of order $k=(2\nu-1)$ \veron{of a Fréchet distribution}. 
\vero{In the same way, we get}
\textcolor{black}{\begin{equation*}\label{TEG9}
\int_{0}^{+\infty}u^{2\nu-2}e^{\frac{-1}{u}\mathcal{T}_1(w,h)}\mathrm{d}u=\mathcal{T}_1(w,h)^{(2\nu-1)}\Gamma(1-2\nu).
\end{equation*}
Then, 
\begin{equation*}\label{TEG10}
\begin{split}
\bE\big[\cD_X^{\nu}(s)\cD_X^{\nu}(s+h)\big]=&\\
\int_0^{+\infty}w^{\nu}\bigg[\mathcal{T}_2(w,h)\mathcal{T}_1(w,h)^{2(\nu-1)}&\Gamma2(\nu-1)+\mathcal{T}_3(w,h)\mathcal{T}_1(w,h)^{(2\nu-1)}\Gamma(1-2\nu)\bigg]\mathrm{d}w\/,
\end{split}
\end{equation*}
}
\vero{and the result follows.}
\end{proof}
  \vero{Corollary \ref{TEG} shows that the risk measure for a TEG process may be computed efficiently, since it reduces to a one dimensional integration involving a Gamma function.}
 \subsection{Behavior of $\cR_1(\lambda\cA,\cD^{\nu}_X)$ with respect to $\lambda$ for max-mixture processes.}
\vero{In what follows, we consider an} isotropic and stationary  max-mixture spatial process with unit Fréchet margin $F$. \vero{We denote $X$ and $V_h^X$ \veron{the} process and the exponent measure function corresponding to the max-stable part and $Y$ and $V_h^Y$ the process and the exponent measure function corresponding to the inverse max-stable process $Y$. Let $a\in [0\/,1]$, $Z= \max(aX\/,(1-a)Y)$. We shall study the behavior of $\cR_1\big(\lambda\cA,\cD^{\nu}_Z\big)$ with respect to $\lambda$. Of course, the case $a=1$ gives results for max-stable processes and $a=0$ gives results for inverse max-stable processes.  Recall that the bivariate distribution function is given by
$$ G_h^Z(x_1\/,x_2) = e^{-aV^X_h(x_1,x_2)}\bigg[e^{\frac{-(1-a)}{x_1}}+e^{\frac{-(1-a)}{x_2}}-1+e^{-V^Y_h(g_a(x_1),g_a(x_2))}\bigg]\/,$$
where $g(z) = -\frac1{\log(1-e^{-\frac1z})}$ and $g_a(z) = g(\frac{z}{1-a})$.\\
Lemma \ref{disk-max} and Proposition \ref{hoff} are a keystone to describe the behaviour of $\cR_1\big(\lambda\cA,\cD^{\nu}_Z\big)$.\\
As in Lemma 3.4 in \cite{ahmed2016spatial}, we get for any $\lambda >0$:} 
\begin{equation}\label{TEGhomo1} 
\cR_1(\lambda\cA,\cD_Z^{\nu})=\int_{h=0}^{\sqrt{2}R}f_{square}(h,R)\mathcal{Q}(\lambda h,\nu)\quad\mathrm{d}h.
\end{equation}
\begin{Co}\label{Prop-max1} 
Let $Z$ be an isotropic and stationary max-mixture spatial process \vero{as above. Assume that the mappings $h\mapsto V_h^X(x_1\/,x_2)$ and \veron{$h\mapsto V_h^Y(x_1\/,x_2)$} are non decreasing for any $(x_1\/,x_2)\in\bR^2_+$.  Let $\cA\subset\ss$ be either a disk or a square, then the mapping $\lambda\mapsto\cR_1(\lambda\cA,\cD_Z^{\nu})$ is non-increasing. }
\end{Co}
\begin{proof} 
\vero{We use (\ref{TEGhomo1}) and from Proposition \ref{hoff}, 
\veron{$$\mathcal{Q}( h,\nu) =\nu^2\int_{0}^{\infty}\int_{0}^{\infty}x_1^{\nu-1}x_2^{\nu-1}\big[G_h^Z(x_1\/,x_2)-F(x_1)F(x_2)\big]\mathrm{d}x_1\mathrm{d}x_2\/.$$}
Since $h\mapsto V_h^X(x_1\/,x_2)$ and \veron{$h\mapsto V_h^Y(x_1\/,x_2)$} are non decreasing, $h\mapsto G_h^Z(x_1\/,x_2)$ is non increasing and the result follows.}
\end{proof}
\vero{\begin{remark}
For a spatial max-stable or inverse max-stable process $X$, the fact that $h\mapsto V_h^X(x_1\/,x_2)$ is non decreasing implies that the dependence between $X(t)$ and $X(t+h)$ decreases as $h$ increases, which seems reasonable in applications. On another hand, if \veron{in addition}, $V_h^X(x_1\/,x_2)$ goes to $\frac1{x_1}+\frac1{x_2}$ as $h$ goes to infinity, this means that $X(t)$, $X(t+h)$ tend to behave independently as $h$ goes to infinity.
\end{remark}}
\begin{Co}\label{AIMM1}
\vero{Let $Z$ be an isotropic and stationary max-mixture spatial process \vero{as above}.  Assume that the mappings $h\mapsto V_h^X(x_1\/,x_2)$ and $h\mapsto V_h^Y(x_1\/,x_2)$ are non decreasing for any $(x_1\/,x_2)\in\bR^2_+$. Moreover, we assume that   
\begin{equation*}
  V_{ h}^X(x_1;x_2) \longrightarrow\frac{1}{x_1}+\frac{1}{x_2} \quad \mathrm{as} \quad h\to\infty
\end{equation*}
and
\begin{equation*}
  V_{h}^Y(x_1,x_2) \longrightarrow\frac{1}{x_1}+\frac{1}{x_2} \quad \mathrm{as} \quad h\to\infty
\end{equation*}
$\forall x_1 , x_2\in\bR_+$.  Let $\cA\subset\ss$ be either a disk or a square, \veron{we have} 
\begin{equation*}
  \lim_{\lambda\to \infty}\cR_1(\lambda\cA,\cD_Z^{\nu})=0.
\end{equation*}
If there exists $V_0$  (resp. $V_1$) an exponent measure function of a non independent max-stable (resp. inverse max-stable) bivariate random vector, such that $ V_{ h}^X \longrightarrow V_0$ \veron{(resp. $V_{ h}^Y \longrightarrow V_1$)} as $h \to \infty$, then
\begin{equation*}
  \lim_{\lambda\to \infty}\cR_1(\lambda\cA,\cD_Z^{\nu})>0.
\end{equation*}
}
\end{Co}
\begin{proof}
\vero{In the case of $A$ a square of side $R$, we use 
\veron{$$\mathcal{Q}( h,\nu) =\nu^2\int_{0}^{\infty}\int_{0}^{\infty}x_1^{\nu-1}x_2^{\nu-1}\big[G_h^Z(x_1\/,x_2)-F(x_1)F(x_2)\big]\mathrm{d}x_1\mathrm{d}x_2\/.$$}
If $  V_{ h}^W(x_1;x_2) $ is non decreasing to $\frac{1}{x_1}+\frac{1}{x_2} $ as  $h\to\infty$ for $W=X$ or $W=Y$, then $G_h^Z(x_1,x_2)$ is non increasing to $F(x_1)F(x_2)$ and we conclude by using the monotone convergence theorem.
}
\end{proof}
\begin{Co}\label{anti-mono}
\vero{Let $Z$ be an isotopic and stationary max-mixture as above.  Assume that $h\mapsto V_{h}^W(x_1\/,x_2)$ is non increasing, with $W=X$ or $W=Y$. Let $\cA_1$ and $\cA_2$ be either disks or squares such that $|\cA_1|\leq|\cA_2|$ then 
\begin{equation*}
\cR_1(\cA_2,\cD_{Z}^\nu)\leq\cR_1(\cA_1,\cD_{Z}^\nu).
\end{equation*}  
}
\end{Co}
\begin{proof}
\vero{
Since the risk measure $\cR_1(\cA\/, \cD_Z^\nu)$ is \veron{invariant} by translation, we may assume that   $\cA_1=\lambda\cA_2$ for some $\lambda\geq 1$. Then, Equation  (\ref{TEGhomo1}) gives the result.
}
\end{proof}
\section{\vero{Numerical} study}\label{R:5}
\vero{In} this section, we will study  the behavior of  the spatial covariance damage  function  and its spatial risk measure  corresponding to a  stationary and isotropic max-stable, inverse max-stable   and max-mixture processes. \vero{We shall use the correlation functions introduced in \cite{abrahamsen1997review}.} 
\begin{enumerate}
  \item Spherical correlation function:
        \begin{equation*}
          \rho^{sph}_{\theta}(h)=\bigg[1-1.5\bigg(\frac{h}{\theta}\bigg)+0.5\bigg(\frac{h}{\theta}\bigg)^3\bigg]\mathds{1}_{\{h>\theta\}}.
        \end{equation*}
  \item Cubic correlation function :
        \begin{equation*}
          \rho^{cub}_{\theta}(h)=\bigg[1-7\bigg(\frac{h}{\theta}\bigg)+\frac{35}{2}\bigg(\frac{h}{\theta}\bigg)^2-\frac{7}{2}\bigg(\frac{h}{\theta}\bigg)^5+\frac{3}{5}\bigg(\frac{h}{\theta}\bigg)^7\bigg]\mathds{1}_{\{h>\theta\}}.
        \end{equation*}
  \item { Exponential correlation functions:
        \begin{equation*}
          \rho^{exp}_{\theta}(h)=\exp\big[-\frac{h}{\theta}\big],
        \end{equation*}}
 \item { Gaussian correlation functions:
        \begin{equation*}
           \rho^{gau}_{\theta}(h)=\exp\big[-\big(\frac{h}{\theta}\big)^2\big];
        \end{equation*}}
 \item {Mat\'ern} correlation function: 
\begin{equation*}
 \rho^{mat}(h)=\frac{1}{\Gamma(\kappa)2^{\kappa-1}}( h/\theta)^{\kappa}K_\kappa(h/\theta)\/,
\end{equation*}
\end{enumerate}
where $\Gamma$ is the gamma function, $K_{\kappa}$ is the modified Bessel function of second kind and order $\kappa>0$, $\kappa$ is a smoothness parameter and $\theta$ is a scaling parameter. 
\subsection{Analysis  \vero{ of the} covariance damage function $\mathcal{Q}(h,\nu)$} 
\vero{The covariance damage function plays a central role in the study of the risk measure $\mathcal{R}_1(\mathcal{A}\/,\mathcal{D}_X^\nu)$.}
\subsubsection{Analysis of $\mathcal{Q}(h,\nu)$ \veron{for max-stable processes}}\label{corre}
We study the behavior of $\mathcal{Q}(h,\nu)$ and $\cR_1(\lambda\cA,\cD^\nu_X)$ \vero{for  $X$ a}  TEG spatial max-stable process, with trunacted parameter $r$, \vero{ non-negative correlation function $\rho$ and} correlation length $\theta$. We \vero{shall denote by $\mathcal{Q}_{\theta,r}(h,\nu)$ the covariance damage function in order to emphasize the dependence of the parameters}. Five different models \vero{with different correlation functions (exponential, Gaussian, spherical, cubic and Matern) introduced above are considered}. \\
 \\
 The behavior of \vero{$\mathcal{Q}_{\theta,r}(h,\nu)$ is  shown} in Figure \ref{MaxF1}.(a). We set the power coefficient  $\nu=0.20$, $r=0.25$ and $\theta=0.20$. \vero{We have that},  $\mathcal{Q}_{\theta,r}(h,\nu)=0$ for any $h\geq2r$; \vero{the decreasing speed changes according to the different  dependence structures.}\\
 \\
For the behavior of $\big(\cD_Y^\nu(s), s\in\ss\big)$ with respect to $\theta$ \vero{is shown in  Figure \ref{MaxF1}(b).} 
   Figure \ref{MaxF1}(c) shows the  behavior of $\mathcal{Q}_{\theta,r}(h,\nu)$  with respect to \vero{the} truncated parameter $r$. We set $\nu=0.20$, $h=0.25$ and $\theta=0.20$.\\
Finally, we study the behavior of the spatial damage covariance function with respect to power coefficient  $\nu$. We set $h=0.25$, $\theta=0.20$ and $r=0.25$.  Figure\ref{MaxF1}.(d) shows  that the covariance  between the damage functions $\cD_Y^\nu(\cdot)$ and $\cD_Y^\nu(\cdot+h)$ \veron{increases} with $\nu$. 
 \begin{figure}[h]
\centering    
\fboxrule=1pt 
\fbox{\includegraphics[width=12cm]{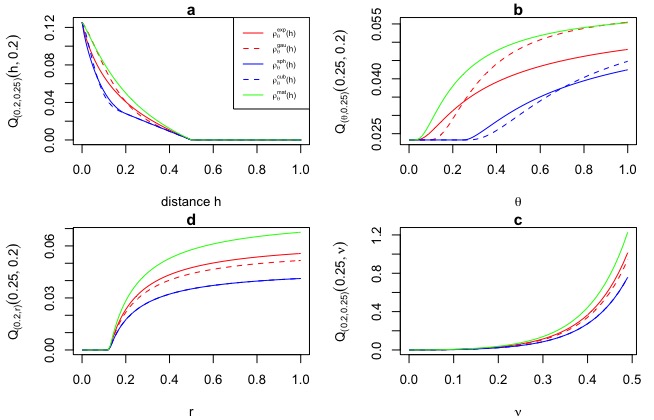}}
\caption{ shows the behavior of $\mathcal{Q}_{\theta,r}(h,\nu)$ with respect to the power coefficient  $\nu$, the correlation length $\theta$, the distance $h$ and truncated parameter $r$. \vero{Plain lines correspond to TEG and dashed lines correspond to inverse TEG.} Five non-negative correlation functions (exponential, Gaussian, spherical, cubic {and Mat\'ern with $\kappa=1$})  have been examined. The graphs (a), (b) ,(c) and (d) show the behavior of $\mathcal{Q}_{\cdot,\cdot}(\cdot,\cdot)$ with respect to: (a)  the distance $h$, when $\nu=0.2$ ,$\theta=0.2$ and $r=0.25$; (b) the correlation length $\theta$, when $\nu=0.2$, $r=0.25$ and $h=0.25$; (c) the truncated parameter $r$, when $\nu=0.2$, $\theta=0.20$ and $h=0.25$; (c) the power coefficient $\nu$, when  $\theta=0.20$, $r=0.25$ and $h=0.25$.}
 \label{MaxF1}
\end{figure}
\begin{remark}
The \vero{global behavior of $\mathcal{Q}_{\theta,r}(h,\nu)$  for an inverse  TEG is the same as for the TEG with the same parameters.}
\end{remark}
\subsubsection{Analysis of $\mathcal{Q}(h,\nu)$ for max-mixture processes}\label{corre12}
\vero{Max-mixture models with TEG max-stable part, denoted $X$ and inverse TEG for the inverse max-stable part - denoted \veron{by} $Y$ - cover} all possible dependence structures in one model (asymptotic dependence  \veron{at} short distances, asymptotic independence \veron{at}  intermediate distances and independence \veron{at} long distances).  \vero{We have simulated five max-mixture  models according to the correlation functions above,}  $X$ and $Y$ have \veron{the} same correlation functions with different correlation lengths. \vero{$r_X$ and $r_Y$ denote the respective truncation parameter of $X$ and $Y$, $\rho_X$ and $\rho_Y$ denote the respective correlation functions of $X$ and $Y$, $\theta_X$ and $\theta_Y$ denote the respective correlation length. The mixing parameter is \veron{denoted} by $a$.}\\
\\
We set \veron{the parameters} $a=0.5$, $r_X=0.15$, $\theta_X=0.10$, $r_Y=0.35$, $\theta_Y=0.3$ and finally $\nu=0.2$. \vero{In this model, the  damage functions $\cD_Y^\nu(\cdot)$ and $\cD_Y^\nu(\cdot+h)$ are asymptotically dependent up to distance $h<2r_X$. The decreasing speed depends on the correlation function, as shown  in Figure \ref{MaxF2}.}
\\ 
\\
\begin{figure}[h]
\centering    
\fboxrule=0pt 
\fbox{\includegraphics[width=12cm]{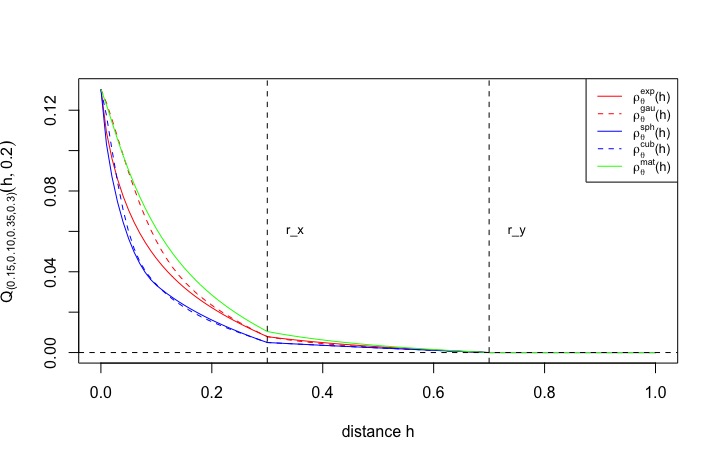}}
\caption{ {shows the behavior of $\mathcal{Q}(h,\nu)$ with respect to \veron{the} distance $h$. Five non-negative correlation functions (exponential, Gaussian, spherical, cubic {and Mat\'ern with $\kappa=1$})  have been examined when $a=0.5$, $\nu=0.2$ and $X$ is \veron{a} TEG max-stable \veron{process} with $\theta_X=0.15$ and $r_X=0.10$; $Y$ is \veron{an inverted TEG process}  with $\theta_Y=0.35$ and $r_Y=0.30$.}}
 \label{MaxF2}
\end{figure}
Figure \ref{MaxF3} shows the behavior of $\mathcal{Q}(h,\nu)$  with respect to each parameter. \vero{When it is not varying, each parameter is  fixed to}  $a=0.5$, $h=0.25$, $\nu=0.2$, $\theta_X=0.1$, $\theta_Y=0.3$, $r_X=0.15$ and $r_Y=0.35$. Graph (a) shows the behavior of $\mathcal{Q}$ with respect to \veron{the mixing parameter}  $a$. The graphs from (b) to (f) shows  the behavior of $\mathcal{Q}$ with respect to \vero{the other parameters. The \veron{behavior} is the same as for max-stable processes.} 
\begin{figure}[h]
\centering    
\fboxrule=1pt 
\fbox{\includegraphics[width=12cm]{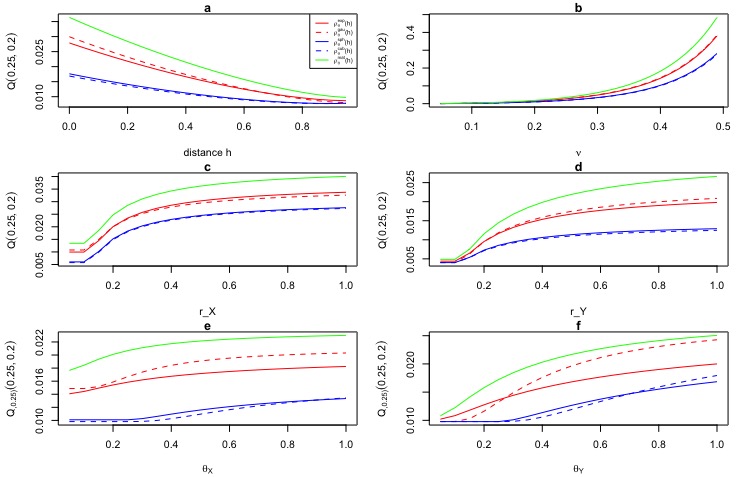}}
\caption{ (a) \veron{shows} the behavior of $\mathcal{Q}(h,\nu)$ with respect to \veron{the} mixing parameter $a$, the power coefficient  $\nu$, the correlation lengths $\theta_X$, $\theta_Y$, and \veron{truncation} parameters $r_X$, $r_Y$. Five non-negative correlation functions (exponential, Gaussian, spherical, cubic {and Mat\'ern with $\kappa=1$})  have been examined.  For $a=0.5$, $h=0.25$, $\nu=0.2$, $\theta_X=0.1$, $\theta_Y=0.3$, $r_X=0.15$ and $r_Y=0.35$, the graphs (a),(b),(c),(d),(e) and (f) show the behavior of $\mathcal{Q}(\cdot,\cdot)$ with respect to: (a)  the mixing parameter  $a$; (b) the power coefficient $\nu$; (c) the \veron{truncation} parameter $r_X$; (d) the truncated parameter $r_Y$; (f) the correlation length $\theta_X$; (e) the correlation length $\theta_Y$.}
 \label{MaxF3}
\end{figure}
 \subsection{Numerical computation of $\cR_1(\cA,\cD^\nu)$}\label{C:1}
 In this study, we \vero{compute} $\cR_1(\cA,\cD^\nu)$ \vero{for different max-stable processes $X$, inverse max-stable processes $Y$ and max-mixture processes $Z$. We considered $X$ a TEG with parameters $r_X$ and $\theta_X$, $Y$ a Smith process with \veron{parameter} $\sigma^2_Y$.  The process $Z$ is a max-mixture between $X$ and $Y$. Max-stable and inverse max-stable models are achieved for $a=1$ and $a=0$, respectively. We compute $\cR_1(\cA,\cD^\nu)$ using (\ref{CTEG}) and (\ref{teg-QQ}) i.e. a $3$ dimensional integration. For these models, the reduction to a one dimensional integration seem not possible. We shall compare this computed value with the Monte Carlo estimation obtained by simulating the process $Z$. In this simulation study, the TEG has parameters:   $r_X=0.25$,  non-negative exponential correlation function with $\theta_X=0.20$. The inverse max-stable $Y$, is given by  a Smith  max-stable process $Y'$ with $\sigma^2_{Y'}=1$. The process $Z$ is simulated } with $n=50$ locations on a grid over a square $\cA=[0,1]^2$. We set the power coefficient $\nu:=\{0.05,0.15,0.25,0.35,0.40\}$ and  mixing parameter $a:=\{0,0.25,0.5,0.75,1\}$.\\
 \\
\vero{The intuitive Monte-Carlo computation (M1),  is obtained by  generating a $m=1000$ sample of $Z$ on the grid.} Then,
$$L_j(\mathcal{A},\mathcal{D}_{Z,u}^+)=\frac{1}{|\mathcal{A}|}\bigg[\frac{R}{n-1}\bigg]^2\sum_{i=1}^{n-1}Z^*(s_{ij})\quad j=1,...,m,
$$
where, $Z^*(s_{ij})=|Z(s_{ij})|^\nu$ 
$$
\bE^{M1}[L(\cA,\cD_{Z}^\nu)]=\frac{1}{m}\sum_{j=1}^{m}L_j(\mathcal{A},\mathcal{D}_{Z}^\nu)
$$
and 
\vero{\begin{equation}\label{simulation}
\mathrm{Var}^{M1}(L(\cA,\cD_{Z}^\nu))=\frac1{m-1}\sum_{j=1}^m(L_j(\mathcal{A},\mathcal{D}_{Z}^\nu)-\bE^{M1}[L(\cA,\cD_{Z}^\nu)])^2\/.
\end{equation}
}
Boxplots in Figure \ref{GauF4}.  represent  the relative errors over $100$ (M1) \vero{simulations}  with respect to the $3$ dimensional integration. \vero{It shows that \veron{the considered risk measures are} hardly estimated by Monte Carlo for $\nu$ greater than $0.30$. \veron{Let us emphasize} that in the $3$ dimensional integration, we used (\ref{teg-QQ}). Using (\ref{TEG-t}) creates numerical issues when $\nu$ approaches $0.4$.}
\begin{figure}[h]
\centering
\fboxrule=0pt 
\fbox{\includegraphics[width=12cm]{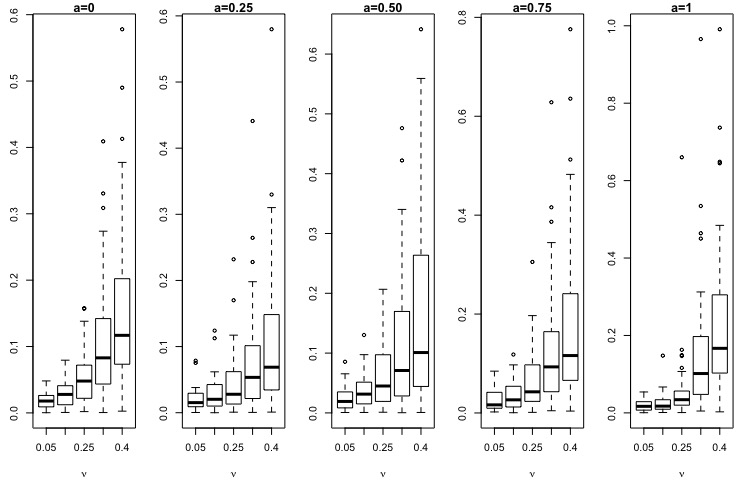}}
\caption{ The boxplots represent the relative errors of \vero{the Monte Carlo estimation of} $\mathrm{Var}(L(\mathcal{A},\mathcal{D}_{Z}^\nu))$ \vero{with respect to the $3$ dimensional integration} for different power coefficient $\nu:=\{0.05,0.15,0.25,0.35,0.40\}$ and  mixing parameter $a:=\{0,0.25,0.5,0.75,1\}$ with parameters $r_X=0.25$ and $\theta_X=0.20$ corresponding to \veron{the} max-stable \veron{process}  $X$ and with $\sigma^2=1$ corresponding to \veron{the inversed Smith process} $Y$  over  a square $\cA=[0,1]^2$.\label{GauF4}}
\end{figure}
\subsection{\vero{Behavior  of } $\cR_1(\lambda\cA,\cD_X^\nu)$}\label{R:beh}
\vero{We are going to study the behavior} of $\cR_1(\lambda\cA,\cD_X^\nu)$ with respect to $\lambda$ \vero{for}  $\cA=[0,1]^2$ \vero{a square and several models. We fixed $\nu=0.20$,  and  $a=0.50$} for max-mixture models and also we will evaluate  $\cR_1(\lambda\cA,\cD_X^\nu)$ with respect to \veron{the} mixing parameter $a$.
 \vero{We considered two models for} $X$: TEG with $r=0.250$ and non-negative exponential correlation function with correlation length $\theta=0.20$;  \veron{Smith} with $\sigma^2=0.6$. \vero{We considered two inverse max-stable processes }$Y$: Inverse TEG and \vero{inverse Smith. The process $Z$ is the max mixture $Z=\max(aX\/,(1-a)Y)$. The different chosen parameters are listed below.}
\begin{itemize}
\item MM1: $X$ is TEG with the \veron{parameters as for the} TEG max-stable \veron{process} above and $Y$ is inverse TEG with $r_Y=0.45$ and non-negative exponential correlation function with correlation length $\theta=0.40$.
\item MM2: $X$ is TEG max-stable with the same \veron{parameters as for } MM1 and  $Y$ is inverse Smith with $\sigma_Y^2=0.8$.
\end{itemize}
\begin{figure}[h]
\centering
\fboxrule=1pt 
\fbox{\includegraphics[width=11cm]{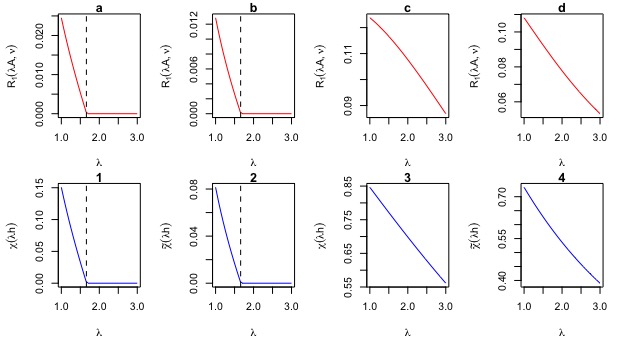}}
\caption{ The graphs represent the behavior of $\cR_1(\lambda\cA,\cD_X^\nu)$ with respect to $\lambda$ \vero{for $\nu=0.20$, a square $\cA=[0,1]^2$ and the } corresponding to tail and lower tail dependence coefficients.  Four models \vero{are considered}:(a) TEG model with truncated parameter $r_X=0.25$ and exponential correlation function with correlation length $\theta_X=0.20$; (b) inverse  TEG max-stable with the \vero{same parameters as in}  (a); (c) Smith max-stable process with $\sigma^2=0.6$; (d) inverse Smith max-stable process with the same \vero{parameters as in} (c). Finally the graphs (1),(2),(3) and (4) represent the  tail and lower tail dependence coefficients corresponding to each model receptively, $h=0.3$.\label{max_invers_max}}
\end{figure}
\vero{Figure \ref{max_invers_max} \veron{is devoted to} max-stable and inverse max-stable processes (no mixture). It shows that  $\cR_1(\lambda\cA,\cD_X^\nu)$ for max-stable and inverse max-stable processes are very similar. Their behavior mimics also the one of $\chi(h)$ in the max-stable case, or $\overline{\chi}(h)$ in the inverse max-stable case. In this picture, we have chosen $h=0.3$.}
\begin{figure}[h]
\centering
\fboxrule=1pt 
\fbox{\includegraphics[width=11cm]{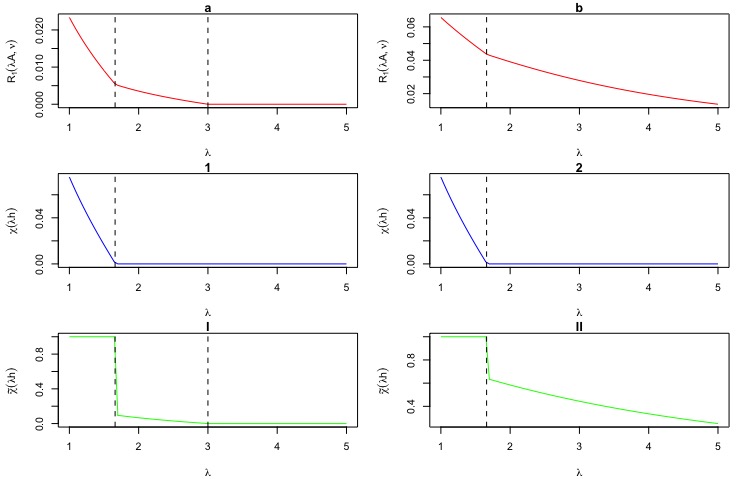}}
\caption{shows the behavior of $\cR_1(\lambda\cA,\cD_X^\nu)$, $\chi(h)$ and $\overline{\chi}(h)$ for two max-mixture models.
\label{MIX_ALL}
}
\end{figure}

Figure \ref{MIX_ALL}.(a) shows the behavior of  $\cR_1(\lambda\cA,\cD_X^\nu)$ \vero{for the} max-mixture model MM1.  It shows a relatively high \vero{value} for $\cR_1(\lambda\cA,\cD_X^\nu)$  up to $0.3\lambda<2r_X$. \vero{Figure \ref{MIX_ALL}.(b) \veron{is devoted to the} model MM2.} \vero{The global behavior is the same for the two models. \veron{We remark that the rupture parameter $r_Y$ is hardly identified on these graphs.} }
 Figure \ref{MIX_ALL}.(b) shows the behavior of  $\cR_1(\lambda\cA,\cD_X^\nu)$ with respect to the max-mixture model MM2. We can see the same behavior of the asymptotic dependence part in MM1 when $0.3\lambda<2r_X$ and \vero{the decrease to  zero from $0.3\lambda\geq2r_X$. The speed of decrease to zero  depends the chosen model. } \\
 \begin{figure}[h]
\centering
\fboxrule=1pt 
\fbox{\includegraphics[width=9cm]{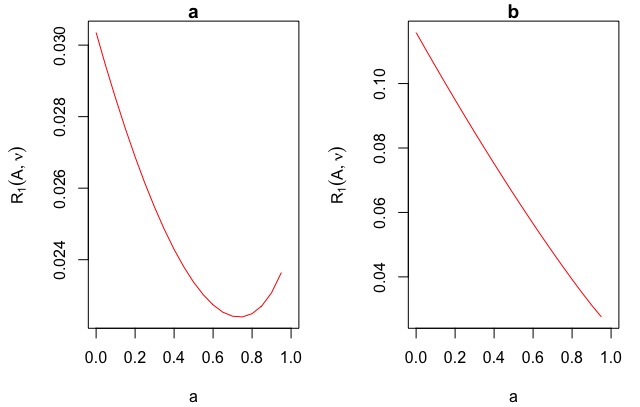}}
\caption{ shows the behavior of $\cR_1(\lambda\cA,\cD_X^\nu)$ with respect to mixing parameter $a$ for two max-mixture models : (a) MM1 model; (b) MM2 model. \label{with_r_a}}
\end{figure}
\\
Figures \ref{with_r_a}. (a) and (b) shows \vero{the behavior of} $\cR_1(\lambda\cA,\cD_X^\nu)$ \vero{with respect to $a$}.  
\section{Conclusion}\label{R:7}
We have \vero{developped the study of the} risk measure  $\cR(\cA,\cD^\nu)$  for spatial \vero{processes allowing asymptotic dependence and asymptotic independence. This risk measure  } takes into account the spatial dependence structure over a region. \vero{It satisfies} the axioms from \cite{ahmed2016spatial} and \cite{koch2015spatial}  for isotropic and stationary max-mixture  processes. A simulation study emphasized the behavior of the risk measure with respect to the various parameters. Finally, the   \vero{sensitivity of spatial risk measures  with different dependence structures is studied  for two different models. }

\ \\
\ \\
\underline{Acknowledgements:} This work was supported by the LABEX MILYON (ANR-10-LABX-0070) of Université de Lyon, within the program "Investissements d'Avenir" (ANR-11-IDEX-0007) operated by the French National Research Agency (ANR).






\bibliographystyle{plain}
\bibliography{sample}

\begin{bibdiv}
\begin{biblist}

\bib{abrahamsen1997review}{book}{
      author={Abrahamsen, P.},
       title={A review of gaussian random fields and correlation functions},
   publisher={Norsk Regnesentral/Norwegian Computing Center},
        date={1997},
}

\bib{ahmed2016spatial}{article}{
      author={Ahmed, M.},
      author={Maume-Deschamps, V.},
      author={Ribereau, P.},
      author={Vial, C.},
       title={Spatial risk measure for gaussian processes},
        date={2016},
     journal={arXiv preprint arXiv:1612.08280},
}

\bib{artzner1999coherent}{article}{
      author={Artzner, P.},
      author={Delbaen, F.},
      author={Eber, J.-M.},
      author={Heath, D.},
       title={Coherent measures of risk},
        date={1999},
     journal={Mathematical finance},
      volume={9},
      number={3},
       pages={203\ndash 228},
}

\bib{bacro2016flexible}{article}{
      author={Bacro, J-N.},
      author={Gaetan, C.},
      author={Toulemonde, G.},
       title={A flexible dependence model for spatial extremes},
        date={2016},
     journal={Journal of Statistical Planning and Inference},
      volume={172},
       pages={36\ndash 52},
}

\bib{coles1999dependence}{article}{
      author={Coles, S.},
      author={Heffernan, J.},
      author={Tawn, J.},
       title={Dependence measures for extreme value analyses},
        date={1999},
     journal={Extremes},
      volume={2},
      number={4},
       pages={339\ndash 365},
}

\bib{davison2012geostatistics}{inproceedings}{
      author={Davison, A.C.},
      author={Gholamrezaee, M.},
       title={Geostatistics of extremes},
organization={The Royal Society},
        date={2012},
   booktitle={Proceedings of the royal society of london a: Mathematical,
  physical and engineering sciences},
      volume={468},
       pages={581\ndash 608},
}

\bib{de1984spectral}{article}{
      author={De~Haan, L.},
       title={A spectral representation for max-stable processes},
        date={1984},
     journal={The annals of probability},
       pages={1194\ndash 1204},
}

\bib{de2007extreme}{book}{
      author={De~Haan, L.},
      author={Ferreira, A.},
       title={Extreme value theory: an introduction},
   publisher={Springer Science \& Business Media},
        date={2007},
}

\bib{donat2011high}{article}{
      author={Donat, M.G.},
      author={Pardowitz, T.},
      author={Leckebusch, G.C.},
      author={Ulbrich, U.},
      author={Burghoff, O.},
       title={High-resolution refinement of a storm loss model and estimation
  of return periods of loss-intensive storms over germany},
        date={2011},
     journal={Natural Hazards and Earth System Sciences},
      volume={11},
      number={10},
       pages={2821\ndash 2833},
}

\bib{follmer2014spatial}{article}{
      author={F{\"o}llmer, H.},
       title={Spatial risk measures and their local specification: The locally
  law-invariant case},
        date={2014},
     journal={Statistics \& Risk Modeling},
      volume={31},
      number={1},
       pages={79\ndash 101},
}

\bib{hougaard2012analysis}{book}{
      author={Hougaard, P.},
       title={Analysis of multivariate survival data},
   publisher={Springer Science \& Business Media},
        date={2012},
}

\bib{keef2009spatial}{article}{
      author={Keef, C.},
      author={Tawn, J.},
      author={Svensson, C.},
       title={Spatial risk assessment for extreme river flows},
        date={2009},
     journal={Journal of the Royal Statistical Society: Series C (Applied
  Statistics)},
      volume={58},
      number={5},
       pages={601\ndash 618},
}

\bib{KOCHErwan2014tools}{thesis}{
      author={Koch, E.},
       title={Tools and models for the study of some spatial and network
  risks:application to climate extremes and contagion in france},
        type={Ph.D. Thesis},
     address={ISFA, University of Claude Bernard Lyon1},
        date={2014},
}

\bib{koch2015spatial}{article}{
      author={Koch, E.},
       title={Spatial risk measures and applications to max-stable processes},
        date={2015},
     journal={To appear in Extremes},
}

\bib{krokhmal2007higher}{article}{
      author={Krokhmal, P.~A.},
       title={Higher moment coherent risk measures},
        date={2007},
     journal={Quantitative Finance},
      volume={7},
      number={4},
       pages={373\ndash 387},
}

\bib{ledford1996statistics}{article}{
      author={Ledford, A.W.},
      author={Tawn, J.~A.},
       title={Statistics for near independence in multivariate extreme values},
        date={1996},
     journal={Biometrika},
      volume={83},
      number={1},
       pages={169\ndash 187},
}

\bib{lehmann1966some}{article}{
      author={Lehmann, E.~L.},
       title={Some concepts of dependence},
        date={1966},
     journal={The Annals of Mathematical Statistics},
       pages={1137\ndash 1153},
}

\bib{moltchanov2012distance}{article}{
      author={Moltchanov, D.},
       title={Distance distributions in random networks},
        date={2012},
     journal={Ad Hoc Networks},
      volume={10},
      number={6},
       pages={1146\ndash 1166},
}

\bib{naveau2009modelling}{article}{
      author={Naveau, P.},
      author={Guillou, A.},
      author={Cooley, D.},
      author={Diebolt, J.},
       title={Modelling pairwise dependence of maxima in space},
        date={2009},
     journal={Biometrika},
      volume={96},
      number={1},
       pages={1\ndash 17},
}

\bib{re2001winterstorms}{article}{
      author={Re, M.},
       title={Winterstorms in europe ii--analysis of 1999 losses and loss
  potentials},
        date={2001},
     journal={Publication of Munich Re},
}

\bib{re2013natural}{article}{
      author={Re, M.},
       title={Natural catastrophes 2012 analyses, assessments, positions 2013
  issue},
        date={2013},
     journal={Topics Geo},
       pages={1\ndash 66},
}

\bib{schlather2002models}{article}{
      author={Schlather, M.},
       title={Models for stationary max-stable random fields},
        date={2002},
     journal={Extremes},
      volume={5},
      number={1},
       pages={33\ndash 44},
}

\bib{sen1994impact}{incollection}{
      author={Sen, P.~K.},
       title={The impact of wassily hoeffding's research on nonparametrics},
        date={1994},
   booktitle={The collected works of wassily hoeffding},
   publisher={Springer},
       pages={29\ndash 55},
}

\bib{smith1990max}{article}{
      author={Smith, R.L.},
       title={Max-stable processes and spatial extremes},
        date={1990},
     journal={Unpublished manuscript, Univer},
}

\bib{tsanakas2003risk}{article}{
      author={Tsanakas, A.},
      author={Desli, E.},
       title={Risk measures and theories of choice},
        date={2003},
     journal={British Actuarial Journal},
      volume={9},
      number={04},
       pages={959\ndash 991},
}

\bib{ulbrich2001three}{article}{
      author={Ulbrich, U.},
      author={Fink, M., A.H.and~Klawa},
      author={Pinto, J.G.},
       title={Three extreme storms over europe in december 1999},
        date={2001},
     journal={Weather},
      volume={56},
      number={3},
       pages={70\ndash 80},
}

\bib{wadsworth2012dependence}{article}{
      author={Wadsworth, J.~A., J. L.and~Tawn},
       title={Dependence modelling for spatial extremes},
        date={2012},
     journal={Biometrika},
      volume={99},
      number={2},
       pages={253\ndash 272},
}

\end{biblist}
\end{bibdiv}


\end{document}